\documentclass[11pt]{amsart}
\usepackage[dvips]{graphicx}
\usepackage{amsmath,graphics}
\usepackage{amsfonts,amssymb}
\usepackage{xypic}
\usepackage{comment}
\usepackage{color}
\usepackage{stmaryrd}
\specialcomment{proofc}{}{}
\includecomment{proofc}
\theoremstyle{plain}
\newtheorem*{theorem*}{Theorem}
\newtheorem*{lemma*} {Lemma}
\newtheorem*{corollary*} {Corollary}
\newtheorem*{proposition*}{Proposition}
\newtheorem*{conjecture*}{Conjecture}
\newtheorem{theorem}{Theorem}[section]
\newtheorem{lemma}[theorem]{Lemma}
\newtheorem*{theorem1*}{Theorem 1}
\newtheorem*{theorem2*}{Theorem 2}
\newtheorem*{theorem3*}{Theorem 3}

\newtheorem{proposition}[theorem]{Proposition}

\theoremstyle{remark}
\newtheorem*{remark}{Remark}

\newtheorem*{note}{Addendum}

\newtheorem{example*}{Example}

\newtheorem*{ack}{Acknowledgement}
\theoremstyle{definition}

   \def\Z{\Bbb{Z}}  \def\C{\Bbb{C}}
\def\N{\Bbb{N}}    
 \def\a{\alpha}   \def\bp{\begin{pmatrix}}
 \def\ep{\end{pmatrix}} \def\bn{\begin{enumerate}} 
   \def\en{\end{enumerate}}
\def\ba{\begin{array}} \def\ea{\end{array}}  
 \def\S{\Sigma}  \def\a{\alpha}

\def\be{\begin{equation}} \def\ee{\end{equation}}

\begin{document}
\title{The Slope of Surfaces with Albanese Dimension One}

\author{Stefano Vidussi}
\address{Department of Mathematics, University of California,
Riverside, CA 92521, USA} \email{svidussi@ucr.edu}
\date{\today}

\begin{abstract} Mendes Lopes and Pardini showed that minimal general type surfaces of  Albanese dimension one have slopes $K^2/\chi$ dense in the interval $[2,8]$. This result  was completed to cover the admissible interval $[2,9]$ by Roulleau and Urzua, who proved that surfaces with fundamental group equal to that of any curve of genus $g \geq 1$ (in particular, having Albanese dimension one) give a set of slopes dense in $[6,9]$. In this note we provide a second construction that complements that of Mendes Lopes--Pardini, to recast a dense set of slopes in $[8,9]$ for surfaces of Albanese dimension one. These surfaces arise as ramified double coverings of cyclic covers of the Cartwright--Steger surface.

\end{abstract}

\maketitle

The basic topological invariants of a minimal irregular surface of general type, $K^2$ and $\chi$, lie in the forward cone in the $(K^2,\chi)$--plane delimited by the Bogomolov--Miyaoka--Yau line $K^2 = 9 \chi$ and the Noether line $K^2 = 2 \chi$. Phrased otherwise, the slope $K^2/\chi$ is contained in the interval $[2,9]$. In \cite{So84} Sommese proved that the set of attainable slopes is dense in $[2,9]$. Several refinements of that result, tied to specific classes of surfaces, have since appeared. In particular Mendes Lopes and Pardini showed in \cite{PML11} that the slope of surfaces of Albanese dimension one is dense in $[2,8]$. Their examples are obtained by taking the double coverings of hyperelliptic genus--$2$ fibrations over a torus, which lie on the line $K^2 = 2 \chi$ and have irregularity $q = 1$, ramified over a collection of general fibers. The proof that slopes of general type surfaces of Albanese dimension one are dense in the entire interval $[2,9]$ was then completed by Roulleau and Urzua in \cite{RU15} (as a corollary of their main result) by showing that  surfaces whose fundamental group is equal to that of a curve of genus $g \geq 1$ have slopes which are dense in $[6,9]$. 

In this note we will provide an alternative construction of minimal surfaces of general type and Albanese dimension one  with slopes dense in the interval $[8,9]$. Our construction follows the strategy of \cite{So84,PML11} (to whom the template should be credited, even if we will be a bit cavalier in punctually referring to), using as building block cyclic covers of the Cartwright--Steger surface and adapting the template to reflect the specifics of this case. The main adjustment involves the identification of an infinite family of unramified cyclic covers of the Cartwright--Steger surface  that retain irregularity $q = 1$.  

We summarize some properties of the Cartwright--Steger surface in the following proposition, referring to \cite{CS10,CKY17} for proof and further details:
\begin{proposition} \label{prop:cs} The Cartwright--Steger surface is a minimal surface of general type with $(K^2,\chi) = (9,1)$, irregularity $q = 1$ and $H_1(X,\Z) = \Z^2$. The Albanese fibration $f \colon X \to T$ has generic fiber $F$ of genus $g(F) = 19$ and no multiple fibers.   \end{proposition}

Our construction will start by taking unramified cyclic covers of the Cartwright--Steger surface. We will need to determine the irregularities (or equivalently the first Betti numbers) of those covers. In order to do so, we will use information on the Green--Lazarsfeld sets of the fundamental group of the Cartwright--Steger surface.
 Recall that the Green--Lazarsfeld sets of the fundamental group $G := \pi_1(X)$ of a K\"ahler variety $X$  are subsets of the character variety ${\widehat G} := Hom(G;\C^{*})$ defined as the collection of the cohomology jumping loci, namely \[ {W}_{i}(G) = \{ \xi \in {\widehat G} ~ | h^1(G;\C_{\xi}) \geq i \},\] nested by the \textit{depth} $i$: ${W}_{i}(G) \subset {W}_{i-1}(G) \subset \ldots \subset {W}_{0}(G) = {\widehat G}$. Moreover, the trivial character ${\hat 1} \in {\widehat G}$ is contained in $W_{i}(G)$ if and only if $i \leq b_{1}(G)$. 

The next proposition is a consequence of the previous, and of structure theorems on the Green--Lazarsfeld sets of a K\"ahler group.

\begin{proposition} \label{prop:gl} The Green--Lazarsfeld sets $W_{i}(G)$ for the Cartwright--Steger surface $X$ are composed of finitely many torsion characters.   \end{proposition}

\begin{proof} By the work of  \cite{Bea92,Si93} (that refined previous results of \cite{GL87,GL91}; see also \cite{De08}), $W_1(G)$ is the union of a finite set of torsion characters and the inverse image of the Green--Lazarsfeld sets of the bases $\S_{\a}$ under the finite collection of fibrations $g_{\a} \colon X \to \S_{\a}$ with base given by a hyperbolic orbisurface $\S_{a}$ of genus $g \geq 1$. If $X$ did admit such fibration, the commutator subgroup of $G$ would be infinitely generated, by \cite{De10}. However, as $H_1(G)= \Z^2$, the commutator subgroup $[G,G]$ of $G$ enters in the short exact sequence \begin{equation} \label{eq:homalb} 1 \longrightarrow [G,G] \longrightarrow G \stackrel{f_{*}}{\longrightarrow} \Z^2 \longrightarrow 1 \end{equation} where $f_{*}$ is the map induced, in homotopy, by the Albanese fibration $f \colon X \to T$. Given any fibration $g \colon X \to \S$, the kernel of the induced map in homotopy $g_* \colon \pi_1(X) \to \pi_1(\S)$ is finitely generated if and only if the fibration has no multiple fibers (see e.g. \cite{Cat03}). By Proposition \ref{prop:cs}, the Albanese fibration has no multiple fibers, hence $[G,G]$ is finitely generated.  It follows that $W_{1}(G)$ is composed of finitely many torsion characters, hence so do the Green--Lazarsfeld sets of higher depth. (In fact, as $X$ has Albanese dimension one, the Albanese fibration is the unique fibration of $X$  having base with positive genus, up to holomorphic automorphism of the base. As this fibration does not have multiple fibers, its base is not hyperbolic as orbifold.) \end{proof}

Proposition \ref{prop:gl} allows one, in principle, to determine the first Betti number of any abelian cover of $X$: Given an epimorphism $\a\colon G \to S$ to a finite abelian group $S$, denote by ${\widehat \a} \colon {\widehat S} \to {\widehat G}$ the induced (injective) map of character varieties. Denote $H := \textit{ker }  \a \leq_{f} G$; Hironaka proved, in \cite{Hi97}, that we have the equation  \begin{equation} \label{eq:hi}  b_1(H) = \sum_{i \geq 1} |W_{i}(G) \cap {\widehat \a}({\widehat S})| = b_1(G) + \sum_{i \geq 1} |W_{i}(G) \cap {\widehat \a}({\widehat S \setminus {\hat 1}})| .\end{equation}  This formula asserts that a character $\xi\colon G \to \C^{*}$ such that $\xi \in W_i(G)$ contributes with multiplicity equal to its depth to the Betti number of the cover defined by $\a\colon G \to S$ whenever it factorizes via $\a$. 

In principle, following \cite{Hi97}, we could explicitly determine the jumping loci $W_{i}(G)$ for the Cartwright--Steger surface out of the Alexander module of its fundamental group using an explicit presentation and Fox calculus. There is no conceptual difficulty in doing so, but it is a somewhat daunting task. We will weasel our way out of this undertaking with the following approach.

First, we will denote by $e(G) \in \N$ the least common multiple of the order of the elements of $W_1(G)$, thought of as elements of the group ${\widehat G}$,  and we will refer to it as the \textit{exponent} of $W_1(G)$ (or, by extension, of the Cartwright--Steger surface). This is a positive integer, well defined by virtue of the fact that $W_1(G)$ is a finite set composed of torsion elements of  ${\widehat G}$. (This notion is borrowed  from the notion of exponent of a finite group, and would coincide with it if $W_1(G)$ were a subgroup of ${\widehat G}$. We don't know whether this is the case.)

\begin{lemma} \label{lem:ec} Let $e(G)$ be the exponent of the Cartwright--Steger surface. Then for every integer $\lambda \geq 0$ all cyclic covers $X_{d}$ of  $X$ of order $d := \lambda e(G) + 1$ have $q(X_d) = q(X) = 1$.  The fibration $f_{d} \colon X_{d} \to T$ induced by the Albanese fibration of $X$ is therefore Albanese. \end{lemma} 
\begin{proof} Let $S :=  \Z_d$  and denote by $\a \colon G\to S$ a cyclic quotient of $G$ of order $d$. Denote by $X_{d}$ the corresponding cyclic cover, so that $\pi_{1}(X_d) = H =  \textit{ker }  \a \leq_{f} G$.
As $H_1(G) = \Z^2$, the quotient map $\a \colon G\to S$ factors through the maximal free abelian quotient of $G$, i.e. the homotopy Albanese map $f_{*} \colon G \to \Z^2$  of  (\ref{eq:homalb}) and  we have the following commutative diagram of fundamental groups
\begin{equation} \label{eq:diag} \xymatrix{
& 1 \ar[d] &
1 \ar[d] & 1 \ar[d] &\\
1 \ar[r] & \Delta \ar[d]^{\cong} \ar[r] &
H \ar[d] \ar[r]& \Z^2 \ar[d] \ar[r]&1\\
1\ar[r]&
[G,G] \ar[r]\ar[d] &
G \ar[d]^{\a}\ar[r]^{f_{*}} & \Z^2 \ar[d] \ar[r]&1\\
 & 1 
 \ar[r]& S \ar[r] \ar[d]  &
S \ar[r] \ar[d] &1 \\
&  & 1 & 1} 
\end{equation}

 By Equation (\ref{eq:hi}) any contribution to $b_{1}(X_d)$ -- hence to $q(X_d)$ -- in excess of $b_1(X) = 2$ comes from nontrivial characters in $W_{1}(G)$ factoring through $\a \colon G \to S$. By definition, the order of the characters in $W_{1}(G)$ divides $e(G)$, while the order of characters in ${\widehat \a}({\widehat S})$ divides $d$. By choosing $d$ to have the form $d = \lambda e(G) + 1$ for $\lambda \geq 0$, $e(G)$ and $d$ are coprime. It follows that $W_{i}(G) \cap {\widehat \a}({\widehat S}) = {\hat 1}$, and the first part of the statement follows. (Note that, if $H_1(X_{d})$ has a nontrivial torsion subgroup, $[H,H] \leq_{f} \Delta \cong [G,G]$ is a {\em proper} subgroup. This has no relevance for us.)  By pull--back, the Albanese fibration of $X$ induces a fibration $f_{d} \colon X_{d} \to T$ (whose induced map in homotopy appears in the top row of the diagram in  (\ref{eq:diag})), where $T$ is an $S$--cover of itself. By the above, this fibration is Albanese. 
\end{proof}

\begin{remark} It would be interesting to know the virtual first Betti number of $X$ or, more modestly, to know the largest Betti number of its finite abelian covers (Proposition \ref{prop:gl} guarantees that the latter number is bounded above). It has been pointed out to us that in \cite{St14} Stover describes a (nonabelian) cover of $X$ with $b_1 = 14$.
Note that the condition $vb_1(X) > 2$ entails that the virtual Albanese dimension of $X$ is $2$. (This latter result has been explicitly observed in \cite{DR17}.)
For the record, using an explicit presentation of $G$ concocted out of information provided in \cite{CKY17} we verified, using GAP, that low index subgroups of $G$ have $b_{1} = 2$. \end{remark}

Now we have information on the first Betti number of ``enough" covers of $X$, and we can proceed as in \cite{So84}. Namely, take a cyclic cover $X_{d}$ of the Cartwright--Steeger surface with $d = \lambda e(G) + 1$ as in  Lemma \ref{lem:ec}. Consider the pull--back fibration $f_{d} \colon X_{d} \to T$, whose generic fiber is isomorphic to $F$, the fiber of the Albanese fibration of $X$. Take the double cover of the base $T$, branched at $2k$ regular values of $f_{d}$. The branch locus uniquely determines a line bundle that we can write as $\mathcal{O}_{T}(p_1+\ldots+p_k)$ where $\{p_1,\ldots,p_k\} \subset T$ are a collection of regular values. Denote by $\Sigma_{k+1} \to T$ the double cover, a surface of genus equal to the subscript. Next, define the fiber product $S_{d,k} :=  X_{d} \times_{T} \Sigma_{k+1}$. Phrased otherwise, $S_{d,k}$ is the double covering of $X_{d}$ determined by the line bundle $f_{d}^{*}\mathcal{O}_{T}(p_1+\ldots+p_k) = \mathcal{O}_{X_d}(F_1+\ldots+F_k)$, where the $F_i = f_{d}^{-1}(p_i)$ are (generic) fibers of the fibration $f_{d} \colon X_{d} \to T$, with branch locus given by $2k$ copies of the generic fiber $F$. Note that $S_{d,k}$ fibers over $\Sigma_{k+1}$.

This construction gives the family of surfaces that we were looking for, as stated by the following theorem that embeds the results above in the template of \cite{So84,PML11}.

\begin{theorem} The surfaces $S_{d,k}$, for $k > 0$ and $d := \lambda e(G) + 1$, $\lambda > 0$, are smooth minimal surfaces of general type with Albanese dimension one whose set of  slopes is dense in the interval $[8,9]$. \end{theorem} 
\begin{proof} The fact that the surfaces  $S_{d,k}$ are minimal of general type with slope \[  \frac{K^2_{S_{d,k}}}{\chi(S_{d,k})} = 9 - \frac{k(g(F) -1)}{2d + k(g(F) - 1) } \in [8,9] \] is discussed in \cite[Lemma 2.1]{So84} (see also \cite{PML11}) -- using information on numerical invariants of branched double coverings that can be found e.g. in  \cite[Section V.22]{BHPV04} --   and is actually true for any choice for $d,k > 0$. 

Next, let's show that  for our choices of $d$ the fibration $S_{d,k} \to \Sigma_{k+1}$ is Albanese. Using the formulae in  \cite[Section V.22]{BHPV04} we have 
\begin{equation} \label{eq:chi} \chi(S_{d,k}) = 2 \chi(X_d) + \frac{1}{2}(\langle kF,K_{X_d}\rangle + \langle kF,kF\rangle) =  2 \chi(X_d) + k(g(F) - 1) \end{equation} where the latter equation follows from the adjunction equality, and 
\begin{equation} \label{eq:pl} p_{g}(S_{d,k}) = p_{g}(X_d) + h^2(\mathcal{O}_{X_d}(-F_1-\ldots-F_k)).\end{equation}
To determine this last term, consider the structure sequence of sheaves for $F_1+\ldots+F_k$, namely 
\[ 0 \longrightarrow \mathcal{O}_{X_d}(-F_1-\ldots-F_k) \longrightarrow \mathcal{O}_{X_d} \longrightarrow \mathcal{O}_{F_1+\ldots+F_k} \longrightarrow 0. \] From the induced long exact sequence in cohomology we can extract the following exact sequence:
\[ \ldots \longrightarrow H^1(\mathcal{O}_{X_d}) \longrightarrow  H^1( \mathcal{O}_{F_1+\ldots+F_k}) \longrightarrow H^2(\mathcal{O}_{X_d}(-F_1-\ldots-F_k)) \longrightarrow H^2(\mathcal{O}_{X_d}) \longrightarrow 0. \]
As $f_{d} \colon X_{d} \to T$ is Albanese, the restriction map $H^1(X_{d},\Z) \to H^{1}(F,\Z)$ is trivial, hence the map $H^1(\mathcal{O}_{X_d}) \to 
H^1( \mathcal{O}_{F_1+\ldots+F_k})$ is the zero map (see e.g. \cite[Lemma IV.12.7]{BHPV04}). Moreover, we have 
$ h^{1}(\mathcal{O}_{F_1+\ldots+F_k}) = kg(F)$ and $h^2(\mathcal{O}_{X_d}) = p_g(X_{d})$. Applying this to Equation (\ref{eq:pl}) we get \[ p_{g}(S_{d,k}) =  2 p_{g}(X_d) + kg(F) \] which, together with Equation (\ref{eq:chi}) gives \[ q(S_{d,k}) = 2q(X_d) + k - 1 = k+1 = q(\Sigma_{k+1}) \] which entails that the fibration $S_{d,k} \to \Sigma_{k+1}$ is Albanese. 

In order to show that the set of slopes achieved by the surfaces $S_{d,k}$ is dense in $[8,9]$, we will need to tweak a bit the calculations of \cite{So84} to take into account the fact that, under our assumptions, not all values of $d$ are allowed. 
Let  $p/q \in (0,1)$ be a rational number, where $p,q$ are positive integers with $0 < p < q $.
The choice of sequences \[ d_{n} =  n  e(G)(q-p) (g(F) -1) + 1,\,\  k_{n} = 2 n e(G)p \] yields surfaces $S_{d_n,k_n}$ where the first index has the form $d = \lambda e(G) + 1$, in particular of Albanese dimension one, and  has the property that  
\[  \lim_n \frac{K^2_{S_{d_n,k_n}}}{\chi(S_{d_n,k_n})} = 9-\lim_n \frac{2ne(G)p(g(F) -1)}{2 n  e(G)(q-p) (g(F) -1) + 2 + 2ne(G)p(g(F) -1)} = 9 - \frac{p}{q}. \]
As the limit set of the slopes contains all rationals in $(8,9)$, it is dense in $[8,9]$
\end{proof} 

\begin{ack} The author would like to thank Rita Pardini and Ziv Ran for useful discussions and for pointing out various inaccuracies. \end{ack}
\begin{note} After completion of this manuscripts, Matthew Stover kindly informed us of his preprint titled {\em ``On general type surfaces with $q = 1$ and $c_2 = 3p_g$"} (\cite{St18}), which contains some results partly overlapping with ours. In particular, Stover shows that the Green--Lazarsfeld sets of the Cartwright--Steger surface are in fact trivial, i.e.  $W_i(G) = \{\hat{1}\}$ for $i \leq 2$, strengthening Proposition \ref{prop:gl} above. This entails that all abelian covers of that surface have irregularity $1$. Moreover, he pointed out that as $G$ is a congruence lattice of positive first Betti number, its virtual Betti number is infinite. Together, these answer the questions posed in the remark following Lemma \ref{lem:ec}  above.  \end{note}


\end{document}